\documentclass[11pt]{article}

\title{From one Reeb orbit to two}
\author{Daniel Cristofaro-Gardiner and Michael
  Hutchings}
\date{}

\usepackage{amssymb}
\usepackage{latexsym}
\usepackage{amsmath}
\usepackage{amsthm}
\usepackage{amscd}

\newcommand{\mc}[1]{{\mathcal #1}}

\numberwithin{equation}{section}

\newtheorem{theorem}{Theorem}[section]
\newtheorem{proposition}[theorem]{Proposition}
\newtheorem{corollary}[theorem]{Corollary}
\newtheorem{lemma}[theorem]{Lemma}
\newtheorem{lemma-definition}[theorem]{Lemma-Definition}

\theoremstyle{definition}

\newtheorem{remark}[theorem]{Remark}

\newcommand{\eqdef}{\;{:=}\;}

\newcommand{\C}{{\mathbb C}}

\newcommand{\R}{{\mathbb R}}
\newcommand{\N}{{\mathbb N}}
\newcommand{\Z}{{\mathbb Z}}

\newcommand{\op}{\operatorname}

\newcommand{\Ker}{\op{Ker}}

\newcommand{\bpm}{\begin{pmatrix}}
\newcommand{\epm}{\end{pmatrix}}

\renewcommand{\epsilon}{\varepsilon}

\begin{document}

\setcounter{tocdepth}{2}

\maketitle

\begin{abstract}
  We show that every (possibly degenerate) contact form on a closed three-manifold has at least two embedded Reeb orbits. We also show that if there are only finitely many embedded Reeb orbits, then their symplectic actions are not all integer multiples of a single real number; and if there are exactly two embedded Reeb orbits, then the product of their symplectic actions is less than or equal to the contact volume of the manifold.  The proofs use a relation between the contact volume and the asymptotics of the amount of symplectic action needed to represent certain classes in embedded contact homology, recently proved by the authors and V.\ Ramos.

\end{abstract}

\section{Statement of results}

Let $Y$ be a closed oriented three-manifold.  Recall that a {\em contact form\/} on $Y$ is a 1-form $\lambda$ on $Y$ such that $\lambda\wedge d\lambda >0$.  A contact form $\lambda$ determines the {\em contact structure\/} $\xi \eqdef \Ker(\lambda)$, and the {\em Reeb vector field\/} $R$ characterized by $d \lambda(R,\cdot)=0$ and $\lambda(R)=1$.  A {\em Reeb orbit\/} is a closed orbit of the vector field $R$, i.e.\ a map $\gamma:\R/T\Z\to Y$ for some $T > 0$ such that $\gamma'(t)=R(\gamma(t))$, modulo reparametrization.  The Reeb orbit $\gamma$ is {\em nondegenerate\/} if the linearized Reeb flow along $\gamma$ does not have $1$ as an eigenvalue, and the contact form $\lambda$ is called nondegenerate if all Reeb orbits are nondegenerate.

The three-dimensional Weinstein conjecture, first proved in full generality by Taubes \cite{tw1}, asserts that any contact form on a closed three-manifold has at least one Reeb orbit.  It is interesting to try to improve the lower bound on the number of Reeb orbits.  In fact, it seems that the only known examples of contact forms on closed three-manifolds with only finitely many embedded Reeb orbits are certain contact forms on $S^3$ and lens spaces with exactly two embedded Reeb orbits, cf.\ \cite[Ex.\ 1.8]{bn}. It is shown in \cite[Thm.\ 1.3]{wh} that any nondegenerate contact form on a closed three-manifold $Y$ has at least two embedded Reeb orbits; and if $Y$ is not $S^3$ or a lens space, then there are at least three embedded Reeb orbits.  The main theorem of the present paper asserts that one can prove the existence of at least two embedded Reeb orbits without the nondegeneracy assumption:

\begin{theorem}
\label{thm:two}
Every (possibly degenerate) contact form on a closed three-manifold has at least two embedded Reeb orbits.
\end{theorem}

For example, Theorem~\ref{thm:two} has the following implication for
Hamiltonian dynamics.  Recall that if $Y$ is a hypersurface in a
symplectic manifold $(X,\omega)$, then the {\em characteristic
  foliation\/} on $Y$ is the rank one foliation $L_Y\eqdef
\op{Ker}(\omega|_{TY})$, and a {\em closed characteristic\/} in $Y$ is
an embedded loop in $Y$ tangent to $L_Y$.  If $Y$ is a regular level
set of a smooth function $H:X\to\R$, then closed characteristics on
$Y$ are the same as unparametrized embedded closed orbits of the
Hamiltonian vector field $X_H$ on $Y$.  Now consider $X=\R^4$ with the
standard symplectic form $\omega=\sum_{i=1}^2dx_i\,dy_i$.  If $Y$ is a
compact hypersurface in $\R^4$ which is star-shaped, meaning that it is tranverse to
the radial vector field, then
\[
\lambda = \frac{1}{2}\sum(x_idy_i-y_idx_i)
\]
restricts to a contact form on $Y$, and the unparametrized embedded Reeb orbits are the same
as the closed characteristics.  (The contact forms that arise this way correspond to the contact forms on $S^3$ for the tight contact structure.) Thus Theorem~\ref{thm:two} applied to $S^3$ implies
the following:
      
\begin{corollary}
\label{cor:dcm}
Every smooth compact star-shaped hypersurface in $\R^4$ has at least two
closed characteristics.
\end{corollary}

Previously, Hofer-Wysocki-Zehnder
showed in \cite[Thm.\ 1.1]{hwz2} that every strictly convex hypersurface
in $\R^4$ has either two or infinitely many closed characteristics,
and in \cite[Cor.\ 1.10]{hwz} that every nondegenerate contact form on
$S^3$ giving the tight contact structure has either two or infinitely
many embedded Reeb orbits, provided that all stable and unstable manifolds
of the hyperbolic periodic orbits intersect transversally.
In higher dimensions, Wang \cite{wang} has shown
that there are at least $\left \lfloor \frac{n+1}{2} \right \rfloor +
1$ closed characteristics on every compact strictly convex
hypersurface $\Sigma$ in $\R^{2n}$.  It has long been conjectured that
there are at least $n$ closed characteristics on every compact convex
hypersurface in $\R^{2n}$, cf.\ \cite[Conj.\ 1]{eh}. After the first version of this paper appeared, alternate proofs of Corollary~\ref{cor:dcm} were given in \cite{ghhm,ll}.

Similarly, applying Theorem~\ref{thm:two} to the unit cotangent bundle of $S^2$ recovers the result of Bangert-Long \cite{bl} that every (not necessarily reversible) Finsler metric on $S^2$ has at least two closed geodesics.

The method used to prove Theorem~\ref{thm:two} yields a slightly more
general result.  To state it, define the
{\em symplectic action\/} of a Reeb orbit $\gamma$ by
\[
\mc{A}(\gamma) \eqdef \int_\gamma\lambda.
\]
We then have:

\begin{theorem}
\label{thm:general}
Let $(Y,\lambda)$ be a closed contact three-manifold having only finitely many embedded Reeb orbits
$\gamma_1,\ldots,\gamma_m$.  Then their symplectic actions
$\mc{A}(\gamma_1),\ldots,\mc{A}(\gamma_m)$ are not all integer
multiples of a single real number.
\end{theorem}

\begin{remark}
  If $\lambda$ has infinitely many embedded Reeb orbits, then their
  symplectic actions can all be integer multiples of a single real
  number, for example in a prequantization space, or in an ellipsoid
  $(\frac{|z_1|^2}{a_1}+\frac{|z_2|^2}{a_2}=1)\subset\C^2$ with $a_1/a_2$ rational.
  Theorem~\ref{thm:general} (and its proof) does extend to contact
  forms with infinitely many embedded Reeb orbits if they are isolated
  in the free loop space.
\end{remark}

To state one more result, if $\lambda$ is a contact form on a closed
oriented three-manifold $Y$, define the {\em volume\/} of
$(Y,\lambda)$ by
\begin{equation}
\label{eqn:contactvolume}
\op{vol}(Y,\lambda) \eqdef \int_{Y}\lambda \wedge d\lambda.
\end{equation}
One can ask whether there exists a Reeb orbit with an upper bound on
its symplectic action in terms of the volume of $(Y,\lambda)$, for example with symplectic action less than or equal to the square root of the volume.  One
might also expect that in most cases there are at least three embedded
Reeb orbits.  The following theorem asserts that at least one of these two statements always holds:

\begin{theorem}
\label{thm:three}
Let $(Y,\lambda)$ be a closed contact three-manifold.  Then either:
\begin{itemize}
\item
$\lambda$ has at least three embedded Reeb orbits, or
\item
$\lambda$ has exactly two embedded Reeb orbits, and their symplectic
actions $T,T'$ satisfy $TT'\le\op{vol}(Y,\lambda)$.
\end{itemize}
\end{theorem}
     
\section{Embedded contact homology and volume}

To prepare for the proofs of Theorem~\ref{thm:two}, \ref{thm:general},
and \ref{thm:three}, we need to recall some notions from embedded
contact homology (ECH).  For more about ECH, see \cite{bn} and the
references therein.

\subsection{Definition of embedded contact homology}
\label{sec:defech}

If $\lambda$ is a nondegenerate contact form on a closed three-manifold $Y$, then for each $\Gamma \in H_1(Y)$ the
{\em embedded contact homology\/} with $\Z/2$ coefficients, which we
denote by $ECH_{*}(Y,\lambda,\Gamma)$, is defined.  (ECH can actually
be defined over $\Z$, see \cite{obg2}, but $\Z/2$ coefficients are
sufficient for the applications in this paper).  This is the homology
of a chain complex $ECC(Y,\lambda,\Gamma,J)$ generated by finite sets
$\alpha=\lbrace(\alpha_i,m_i)\rbrace$ such that the $\alpha_i$ are distinct embedded Reeb orbits, $m_i=1$ when $\alpha_i$ is hyperbolic, and
\[
\sum_i m_i[\alpha_i]=\Gamma \in H_1(Y).
\]
Here a Reeb orbit $\gamma$ is called {\em hyperbolic\/} if the linearized
Reeb flow around $\gamma$ has real eigenvalues.
The $J$ that enters into the chain complex is an
$\R$-invariant almost complex structure on $\R \times Y$ that sends
the two-plane field $\xi=\Ker(\lambda)$ to itself, rotating it positively with
respect to $d\lambda$, and satisfies $J(\partial_s)=R$, where $s$
denotes the $\R$ coordinate on $\R\times Y$.  The chain complex
differential $\partial$ counts certain mostly embedded $J$-holomorphic
curves in $\R \times Y$. Specifically, if $\alpha$ and $\beta$ are two
chain complex generators, then the differential coefficient
$\langle \partial \alpha, \beta \rangle\in\Z/2$ is a count of
$J$-holomorphic curves in $\R \times Y$, modulo translation of the
$\R$ coordinate, that are asymptotic as currents to $\R \times \alpha$
as $s \to \infty$ and to $\R \times \beta$ as $s \to -\infty$.  The
curves are required to have {\em ECH index\/} $1$.  The ECH index is a
certain function of the relative homology class of the curve,
explained e.g.\ in \cite{ir}; we do not need to recall the definition
here.  If $J$ is generic, then $\partial$ is well-defined and
$\partial^2=0$, as shown in \cite{obg1,obg2}.

The ECH index induces a relative $\Z/d$-grading on
$ECH_{*}(Y,\lambda,\Gamma)$, where $d$ denotes the divisibility of
$c_1(\xi)+2\op{PD}(\Gamma)$ in $H^2(Y;\Z)$ mod torsion, see
\cite[\S2.8]{ir}.  Here $\op{PD}(\Gamma)$ denotes the Poincare dual
of $\Gamma$.

\subsection{The isomorphism with Seiberg-Witten Floer cohomology}

Although a priori the homology of the chain complex
$ECC(Y,\lambda,\Gamma,J)$ might depend on $J$, in fact it does not.
This follows from a theorem of Taubes \cite{e1} asserting
that when $Y$ is connected, there is a canonical isomorphism between
embedded contact homology and a version of Seiberg-Witten Floer
cohomology as defined by Kronheimer-Mrowka \cite{km}.  The precise statement is that there is a canonical
isomorphism of relatively graded $\Z/2$-modules
\begin{equation}
\label{eqn:echsw}
ECH_{*}(Y,\lambda,\Gamma) \simeq
\widehat{HM}^{-*}(Y,\mathfrak{s}_{\xi}+PD(\Gamma)),
\end{equation}
where $\mathfrak{s}_{\xi}$ is the spin-c structure determined by the
oriented two-plane field $\xi$, see e.g.\ \cite[Lem.\ 28.1.1]{km}.
(The isomorphism \eqref{eqn:echsw} holds using $\Z$ or $\Z/2$ coefficients.)  In particular, there is a
well-defined relatively graded $\Z / 2$-module $ECH(Y,\xi,\Gamma)$.
By summing over all $\Gamma\in H_1(Y)$, one also obtains a
well-defined relatively graded $\Z /2$-module $ECH(Y,\xi)$.

\subsection{Filtered ECH}
\label{sec:FECH}

If $\alpha=\{(\alpha_i,m_i)\}$ is a generator of the ECH chain complex, define the {\em symplectic action\/} of $\alpha$ by

\[
\mc{A}(\alpha) \eqdef \sum_im_i\mc{A}(\alpha_i) = \sum_im_i \int_{\alpha_i}\lambda.
\]
It follows from the conditions on $J$ that the ECH differential
decreases the symplectic action.  Hence, for any real number $L$, one
can define the {\em filtered ECH\/}, denoted by $ECH^{L}(Y,\lambda,\Gamma)$, to be
the homology of the subcomplex of $ECC$ spanned by
generators with action strictly less than $L$.

It is shown in \cite[Thm.\ 1.3]{cc2} that $ECH^{L}(Y,\lambda,\Gamma)$
does not depend on the choice of generic $J$ required to define the
chain complex differential.  On the other hand,
$ECH^{L}(Y,\lambda,\Gamma)$, for fixed $Y$, $\Gamma$ and $L$, does depend
on the contact form $\lambda$ and not just on the contact structure
$\xi$.

As with the usual ECH, one can take the direct sum of the filtered ECH for all $\Gamma\in H_1(Y)$
to obtain a relatively graded $\Z / 2$ module $ECH^{L}(Y,\lambda)$.

\subsection{The U map}

If $Y$ is connected, there is a degree $-2$ map
\begin{equation}
\label{eqn:Umap}
U: ECH(Y,\lambda,\Gamma) \to ECH(Y,\lambda,\Gamma).
\end{equation}
It is induced by a chain map $U_z$ which is defined similarly to the
differential $\partial$, but instead of counting ECH index $1$ curves
modulo translation, it counts $J$-holomorphic curves of ECH index $2$
passing through $(0,z) \in \R \times Y$, where $z$ is a base point in
$Y$ which is not contained in any Reeb orbit.  The connectedness of $Y$ implies that the induced map \eqref{eqn:Umap}
does not depend on $z$.  (When $Y$ is disconnected there is one $U$
map for each component.)  For details see \cite[\S2.5]{wh} or \cite[\S3.8]{bn}.

There is an analogous $U$ map on Seiberg-Witten Floer cohomology, and
it is shown in \cite[Thm.\ 1.1]{e5} that this agrees with the $U$ map
on ECH under the isomorphism \eqref{eqn:echsw}.
 
\subsection{Minimum symplectic action needed to represent a class}

Let $0\neq \sigma \in ECH(Y,\xi)$.  We now recall from \cite{qech} the
definition of a real number $c_{\sigma}(Y,\lambda)$, which roughly
speaking is the minimum symplectic action needed to represent the
class $\sigma$.

If $\lambda$ is nondegenerate, then $c_{\sigma}(Y,\lambda)$ is the
infimum over $L$ such that $\sigma$ is in the image of the
inclusion-induced map $ECH^L(Y,\lambda)\to ECH(Y,\xi)$.  Note that for
any $J$ as needed to define the chain complex $ECC(Y,\lambda,J)$,
there exists a cycle $\theta$ in the chain complex representing the
class $\sigma$, such that every chain complex generator $\alpha$ that
appears in $\theta$ satisfies $\mc{A}(\alpha) \le
c_\sigma(Y,\lambda)$, and $c_\sigma(Y,\lambda)$ is the smallest number
with this property.  We call a cycle $\theta$ as above an {\em
  action-minimizing representative\/} of $\sigma$.

If $\lambda$ is degenerate, one defines
\begin{equation}
\label{eqn:degenerate}
c_\sigma(Y,\lambda)=\lim_{n\to\infty}c_\sigma(Y,f_n\lambda),
\end{equation}
where $f_n:Y\to\R$ are positive smooth functions such that the contact form $f_n\lambda$ is
nondegenerate and $\lim_{n\to\infty}f_n=1$ in the $C^0$ topology.

The numbers $c_\sigma(Y,\lambda)$ then satisfy the following axioms:
\begin{description}
\item{(Monotonicity)} If $f:Y\to\R$ is a smooth function with $f>1$,
    then $c_\sigma(Y,\lambda)\le c_\sigma(Y,f\lambda)$.
\item{(Scaling)} If $\kappa>0$ is a constant then
  $c_\sigma(Y,\kappa\lambda)=\kappa c_\sigma(Y,\lambda)$.
\item{(Continuity)} If $f_n:Y\to\R$ are positive smooth functions with
  $\lim_{n\to\infty}f_n=1$ in the $C^0$ topology, then
  $\lim_{n\to\infty}c_\sigma(Y,f_n\lambda)=c_\sigma(Y,\lambda)$.
\end{description}
To see that \eqref{eqn:degenerate} is well-defined and to prove the
above axioms, one can first show that the Monotonicity and Scaling axioms
hold for nondegenerate contact forms, see \cite[\S4]{qech}.  It then
follows from this that the definition \eqref{eqn:degenerate} does not depend on
the sequence $\{f_n\}$, and that the Monotonicity, Scaling,
and Continuity axioms hold without any nondegeneracy assumption.

\subsection{Asymptotics and volume}
\label{sec:volume}

In \cite{vc}, the following result was established relating the asymptotics of the numbers $c_{\sigma}(Y,\lambda)$
 to the contact volume \eqref{eqn:contactvolume}.
  If $\Gamma \in H_1(Y)$ is such that $c_1(\xi)+2PD(\Gamma)
\in H^2(Y;\Z)$ is torsion, then we know from \S\ref{sec:defech} that
$ECC(Y,\xi,\Gamma)$ has a relative $\Z$-grading.  Choose any
normalization of this to an absolute $\Z$-grading, and denote the
grading of a generator $x$ by $I(x) \in \Z$.  We then have:

\begin{theorem}
\label{thm:strongvc}
\cite[Thm.\ 1.3]{vc}
Let $(Y,\lambda)$ be a closed connected contact three-manifold, let
$\Gamma \in H_1(Y)$, suppose that $c_1(\xi)+2PD(\Gamma)
\in H^2(Y,\Z)$ is torsion, and choose an absolute $\Z$-grading $I$ on $ECH(Y,\xi,\Gamma)$.  Let $\lbrace \sigma_k
\rbrace_{k=1,2,\ldots}$ be a sequence of nonzero homogeneous elements of
$ECH(Y,\xi,\Gamma)$ satisfying $\lim_{k \to
  \infty} I(\sigma_k) = \infty$.  Then
\begin{equation}
\label{eqn:volume}
\lim_{k \to \infty} \frac{c_{\sigma_k}(Y,\lambda)^2}{I(\sigma_k)}=\op{vol}(Y,\lambda).
\end{equation}
\end{theorem}

To prove Theorem~\ref{thm:two} and Theorem~\ref{thm:general}, we just need the
following weaker result:

\begin{corollary}
\label{cor:wvc}
  Let $(Y,\lambda)$ be a closed connected contact three-manifold.  Then there
  exist nonzero classes $\{\sigma_k\}_{k\ge 1}$ in $ECH(Y,\xi)$ such that
\begin{equation}
\label{eqn:*1}
U\sigma_{k+1} = \sigma_k
\end{equation}
for all $k\ge 1$, and
\begin{equation}
\label{eqn:*2}
\lim_{k\to\infty}\frac{c_{\sigma_k}(Y,\lambda)}{k}=0.
\end{equation}.
\end{corollary}

\begin{proof}  We can always find a class $\Gamma\in H_1(Y)$ such
that $c_1(\xi)+2\op{PD}(\Gamma)\in H^2(Y;\Z)$ is torsion.  It follows
from the isomorphism \eqref{eqn:echsw} of $ECH(Y,\xi,\Gamma)$ with
Seiberg-Witten Floer cohomology, together with known properties of the
latter \cite[Lem. 33.3.9, Cor.\ 35.1.4]{km}, that there exists a
sequence $\{\sigma_k\}_{k\ge 1}$ of nonzero homogeneous elements of
$ECH(Y,\xi,\Gamma)$ satisfying
\eqref{eqn:*1}. Since the $U$ map has degree $-2$, we have
$I(\sigma_{k+1})=I(\sigma_k)+2$.  Hence, Theorem~\ref{thm:strongvc} applies
to give \eqref{eqn:volume}, which then implies \eqref{eqn:*2}.
\end{proof}

\begin{remark}
\label{rem:wvc}
The analysis in \cite{vc} is not required for Corollary 1.2, because it was already shown in \cite{qech} that Theorem~\ref{thm:strongvc}  holds for any contact form on $S^3$ giving the tight contact structure.  In particular, it follows from \cite[Rmk. 3.3, Prop. 4.5]{qech} that Theorem~\ref{thm:strongvc} holds for the boundary of an ellipsoid in $\R^4$, and it then follows from \cite[Prop. 8.6(b)]{qech} that Theorem~\ref{thm:strongvc} holds for any other contact form giving the same contact structure.
\end{remark}

\section{The key lemma}

The key to the proofs of Theorems~\ref{thm:two}, \ref{thm:general},
and \ref{thm:three} is the following:

\begin{lemma}
\label{lem:key}
Let $Y$ be a closed connected three-manifold and let $\lambda$ be a
(possibly degenerate) contact form on $Y$ with kernel $\xi$.  Assume
that $\lambda$ has only finitely many embedded Reeb orbits
$\gamma_1,\ldots,\gamma_m$.  Then:
\begin{description}
\item{(a)} If $0\neq \sigma\in ECH(Y,\xi)$, then
  $c_\sigma(Y,\lambda)$ is a nonnegative integer linear combination of
  $\mc{A}(\gamma_1),\ldots,\mc{A}(\gamma_m)$.
\item{(b)} If $\sigma\in ECH(Y,\xi)$ and $U\sigma\neq 0$, then
  $c_{U\sigma}(Y,\lambda) < c_\sigma(Y,\lambda)$.
\end{description}
\end{lemma}

\begin{proof}
  Fix a nonzero class $\sigma\in ECH(Y,\xi)$ and write
  $L=c_\sigma(Y,\lambda)$.  Choose open tubular neighborhoods $N_i$ of
  the Reeb orbits $\gamma_i$ whose closures are disjoint, and let
  $N=\bigcup_{i=1}^m N_i$.  Fix a point $z\in Y \setminus
  \overline{N}$ for use in defining the $U$ map.  By shrinking the
  tubular neighborhoods $N_i$ if necessary, we may assume that:
\begin{description}
\item{(i)} If $\gamma$ is a Reeb trajectory intersecting both $z$ and
  $\overline{N}$ then $\int_\gamma\lambda\ge L+3$.
\end{description}
Next, choose a sequence of smooth functions $\{f_n:Y\to\R^{>0}\}$ such
that:
\begin{description}
\item{(ii)}
 $f_n|_{Y\setminus N}\equiv 1$,
\item{(iii)}
The contact form
  $f_n\lambda$ is nondegenerate,
\item{(iv)}
$\lim_{n\to\infty}f_n=1$ in the
  $C^1$ topology, and
\item{(v)}
Every Reeb orbit of $f_n\lambda$ with
  symplectic action less than $L+1$ is contained in some $N_i$, and
  has symplectic action within $1/n$ of an integer multiple of
  $\mc{A}(\gamma_i)$.
\end{description}
(The reason we can obtain condition (v) is that otherwise there would
be a sequence $f_n$ such that each $f_n\lambda$ has a Reeb orbit of
action less than $L+1$ not contained in $N$, or a Reeb orbit in $N_i$
of action $<L+1$ whose action is not within $\epsilon$ of an integer
multiple of $\mc{A}(\gamma_i)$ for some $n$-independent
$\epsilon>0$. Then a subsequence of these Reeb orbits would converge
to a Reeb orbit of $\lambda$ which could not be a multiple of one of
the Reeb orbits $\gamma_i$.)

It follows from conditions (iii) and (v) that $c_\sigma(Y,f_n\lambda)$
is within distance $m/n$ of an integer linear combination of
$\mc{A}(\gamma_1),\ldots,\mc{A}(\gamma_m)$.  Assertion (a) of the
lemma now follows from condition (iv) and the Continuity axiom for
$c_\sigma$.

To prove (b), continue to fix the above data, and assume that
$U\sigma\neq 0$.  For each $n$, choose a generic almost complex
structure $J_n$ on $\R\times Y$ as needed to define the filtered ECH
chain complex $ECC^{L+1}(Y,f_n\lambda,J_n)$ and the chain map $U_z$ on
it.  Specifically, we need $J_n$ to satisfy the genericity conditions
listed in the first paragraph of \cite[\S10]{obg2}, for
$J_n$-holomorphic curves counted by $\partial$ or $U_z$ whose positive
ends have total action less than $L+1$.  These conditions on $J_n$ can
all be achieved by perturbing $J_n$ near the Reeb orbits of action
less than $L+1$.  So by condition (v) above, we can arrange that the
almost complex structures $J_n$ agree with a fixed almost complex
structure $J_0$ on $\R\times (Y\setminus N)$.

We know from the proof of (a) that if $n$ is sufficiently large then
$c_\sigma(Y,f_n\lambda)<L+1$, so we can choose an action-minimizing
representative $\theta_n$ of $\sigma$ in $ECC^{L+1}(Y,f_n\lambda)$.

\medskip

{\em Claim.\/} There exists $\delta>0$ and a positive integer $n_0$
such that if $n\ge n_0$ and $C_n$ is a $J_n$-holomorphic curve counted
by $U_z\theta_n$, then $\int_{C_n}d(f_n\lambda)\ge \delta$.

\medskip

The Claim implies (b), because it implies that if $n\ge n_0$ then
$c_{U\sigma}(Y,f_n\lambda)\le c_\sigma(Y,f_n\lambda) - \delta$, and so
by the Continuity axiom $c_{U\sigma}(Y,\lambda)\le
c_\sigma(Y,\lambda)-\delta$.

\medskip

{\em Proof of Claim:\/} Recall that the conditions on $J_n$ imply that
if $C_n$ is any $J_n$-holomorphic curve, then $d(f_n\lambda)$ is
pointwise nonnegative on $C_n$, with equality only where the tangent
space to $C_n$ is the span of the $\R$ direction and the Reeb
direction (or where $C_n$ is singular, although below $C_n$ will be a curve counted by $U_z\theta_n$ and these do not have singularities).  In particular,
$\int_{C_n}d(f_n\lambda)\ge 0$.  Consequently, if the Claim is false,
then we can find an increasing sequence $\{n_i\}_{i\ge 1}$ of positive
integers, and for each $i$ a $J_{n_i}$-holomorphic curve $C_{n_i}$
counted by $U_z\theta_{n_i}$, such that
$\lim_{i\to\infty}\int_{C_{n_i}}d(f_{n_i}\lambda)=0$.

We now use the following proposition, which is a special case of a
result of Taubes \cite[Prop.\ 3.3]{t}:

\begin{proposition}
\label{prop:taubes}
Let $(X,\omega)$ be a compact symplectic $4$-manifold with boundary
with a compatible almost complex structure $J$.  Let $\{C_{i}\}_{i \in
  \N}$ be a sequence of compact $J$-holomorphic curves in $X$ with
boundary contained in $\partial X$, and suppose that there exists $E >
0$ such that $\int_{C_i} \omega < E$ for all $i$.  Then one can pass
to a subsequence such that:
\begin{description}
\item{(Convergence as currents)} The curves $\{C_{i}\}$ converge
  weakly as currents to a compact $J$-holomorphic curve $C_0$ with
  boundary in $\partial X$ such that $\int_{C_0} \omega\le E$, and
\item{(Pointwise convergence)}
\[
\lim_{i\to\infty} \left(\sup_{x\in C_{i^*}}\op{dist}(x,C_0) +
  \sup_{x\in C_0}\op{dist}(x,C_{i^*})\right) = 0.
\]
\end{description}
\end{proposition}

We apply the above proposition to the
intersections of the holomorphic curves $C_{n_i}$ with $X=[-1,1]\times
(Y\setminus N)$, with the symplectic form $\omega =
d(e^s\lambda)$.  To see why we have the necessary upper bound on
$\omega$ to apply the proposition, given $i$, choose $s_+\in[1,2]$ and
$s_-\in[-2,-1]$ such that $C_{n_i}$ is tranverse to $\{s_\pm\}\cap Y$.
 Then since $d(e^{s}f_{n_i}\lambda)$ and $d(f_{n_i}\lambda)$
are pointwise nonnegative on $C_{n_i}$, we have an upper bound
\[
\begin{split}
\int_{C_{n_i}\cap([-1,1]\times (Y\setminus N))}\omega & \le
\int_{C_{n_i}\cap ([s_-,s_+]\times Y)}d(e^sf_{n_i}\lambda)\\
& = e^{s_+}\int_{C_{n_i}\cap (\{s_+\}\times Y)}f_{n_i}\lambda -
e^{s_-}\int_{C_{n_i}\cap (\{s_-\}\times Y)}f_{n_i}\lambda \\
& < e^2(L+1).
\end{split}
\]
So we can pass to a subsequence such that $C_{n_i}\cap([-1,1]\times
(Y\setminus N))$ converges in the sense of
Proposition~\ref{prop:taubes} to a (possibly multiply covered)
$J_0$-holomorphic curve $C_0$ in $[-1,1]\times (Y\setminus N)$.  By
the ``pointwise convergence'' condition, the curve $C_0$ contains the
point $(0,z)$, since each $C_{n_i}$ does.

Since $C_0$ is $J_0$-holomorphic, it follows that $d\lambda$ is
pointwise nonnegative on $C_0$, with equality only where $C_0$ is
singular or the tangent space of $C_0$ is the span of the $\R$
direction and the Reeb direction.  In particular,
\begin{equation}
\label{eqn:strict}
\int_{C_0}d\lambda\ge 0.
\end{equation}

In fact, the inequality \eqref{eqn:strict} must be strict.  Otherwise
$C_0$, regarded as a current, is invariant under translation of the
$[-1,1]$ coordinate on $[-1,1]\times(Y\setminus N)$.  It follows that
$C_0\cap(\{0\}\times (Y\setminus N))$ is tangent to the Reeb vector
field for $\lambda$.  In particular, $C_0\cap(\{0\}\times (Y\setminus
N))$, regarded as a subset of $Y$, contains a Reeb trajectory for
$\lambda$ passing through $z$ with endpoints on
$\partial\overline{N}$.  So by (i) above,
\[
\int_{C_0\cap(\{0\}\times(Y\setminus N))}\lambda\ge L+3.
\]
By the convergence of currents above, it follows that
\begin{equation}
\label{eqn:cc}
\int_{C_{n_i}\cap(\{s\}\times (Y\setminus N))}f_{n_i}\lambda \ge L+2
\end{equation}
whenever $i$ is sufficiently large and $s\in[-1,1]$ is such that
$C_{n_i}$ is transverse to $\{s\}\times Y$.  When this transversality
holds, we orient $C_{n_i}\cap(\{s\}\times Y)$, regarded as a
submanifold, by the ``$\R$-direction first'' convention.  The
conditions on $J_{n_i}$ imply that $f_{n_i}\lambda$ is pointwise
nonnegative on this oriented one-manifold, so it follows from
\eqref{eqn:cc} that
\begin{equation}
\label{eqn:cc2}
\int_{C_{n_i}\cap(\{s\}\times Y)}f_{n_i}\lambda \ge L+2.
\end{equation}
But this is impossible, because the left hand side of \eqref{eqn:cc2}
must be less than or equal to the maximum symplectic action of a
generator in $\theta_{n_i}$, which is less than $L+1$.  This
contradiction proves that the inequality \eqref{eqn:strict} is strict.

Given this, let $\delta=\frac{1}{2}\int_{C_0}d\lambda>0$.  It follows
from the convergence of currents that if $i$ is sufficiently large
then
\[
\begin{split}
\int_{C_{n_i}}d(f_{n_i}\lambda) & \ge
\int_{C_{n_i}\cap([-1,1]\times (Y\setminus N))}d(f_{n_i}\lambda)\\
& = \int_{C_{n_i}\cap([-1,1]\times (Y\setminus N))}d\lambda\\
&\ge \int_{C_0}d\lambda\;-\delta\\
& = \delta.
\end{split}
\]
This contradicts our assumption that $\lim_{i\to
  \infty}\int_{C_{n_i}}d(f_{n_i}\lambda)=0$ and thus completes the proof of
the Claim, and with it Lemma~\ref{lem:key}.
\end{proof}

\begin{remark}
\label{rmk:whycurrents}
In the above argument we can not quote the SFT compactness theorem
from \cite{behwz}, because that result assumes both a genus bound
(which one does not have in ECH) as well as nondegeneracy of the
contact form.  This is why we use Taubes's approach via currents.
Although this is only applicable in four dimensions, if one has a
genus bound then one can cite \cite{fish} for similar arguments in
higher dimensions.
\end{remark}

\section{Proofs of theorems}

\begin{proof}[Proof of Theorem~\ref{thm:two}.] This follows from
  Theorem~\ref{thm:general}.
\end{proof}

\begin{proof}[Proof of Theorem~\ref{thm:general}.] Suppose that
  $\lambda$ has only finitely many embedded Reeb orbits
 and suppose that their symplectic
  actions are all integer
  multiples of a single real number $T>0$.  Let $\{\sigma_k\}_{k\ge
    1}$ be any sequence satisfying \eqref{eqn:*1}.  Then by
  Lemma~\ref{lem:key}, we have $c_{\sigma_k}(Y,\lambda)=n_kT$ where
  $\{n_k\}_{k\ge 1}$ is a strictly increasing sequence of nonnegative
  integers.  It follows that
\begin{equation}
\label{eqn:liminf}
\liminf_{k\to\infty}\frac{c_{\sigma_k}(Y,\lambda)}{k} \ge T,
\end{equation}
so that \eqref{eqn:*2} cannot hold. This contradicts Corollary~\ref{cor:wvc}. 
\end{proof}

\begin{proof}[Proof of Theorem~\ref{thm:three}.]
  Suppose there are fewer than three embedded Reeb orbits.  We know
  from Theorem~\ref{thm:two} that $Y$ is connected and there are exactly two embedded Reeb
  orbits; denote their symplectic actions by $T$ and $T'$.

Let $\{\sigma_k\}_{k\ge 1}$ be a sequence of homogeneous classes satisfying \eqref{eqn:*1}.  By Lemma~\ref{lem:key}, we have
  $c_{\sigma_k}(Y,\lambda) = n_kT+n_k'T'$ where $n_k$ and $n_k'$ are
  nonnegative integers such that $n_{k+1}T+n_{k+1}'T'>n_kT+n_k'T'$.
  It follows from this that
\begin{equation}
\label{eqn:short}
\lim_{k\to\infty}\frac{c_{\sigma_k}(Y,\lambda)^2}{k} \ge 2TT'.
\end{equation}
To see this, note that if we fix $k$ and write
$L=c_{\sigma_k}(Y,\lambda)=n_kT+n_{k'}T'$, then $k$ is less than or
equal to the number of pairs of nonnegative integers $(x,y)$ with
$xT+yT'\le L$, which is the number of lattice points in the triangle
enclosed by the line $Tx+T'y=L$ and the $x$ and $y$ axes, which is
$L^2/(2TT')+O(L)$, compare \cite[\S3.3]{qech}.  On the other hand, since
the $U$ map has degree $-2$, we have
\begin{equation}
\label{eqn:short2}
\lim_{k\to\infty}\frac{I(\sigma_k)}{k} = 2.
\end{equation}
Putting \eqref{eqn:short} and \eqref{eqn:short2} into
\eqref{eqn:volume} gives $\op{vol}(Y,\lambda)\ge TT'$.
\end{proof}

\paragraph{Acknowledgments} The first author was partially supported
by NSF grant DMS-0838703.  The second author was partially supported
by NSF grant DMS-0806037.

\end{document}